\documentclass[12pt,leqno,a4paper]{amsart}
\usepackage{amssymb, amsmath, enumerate}
\usepackage{enumerate,color}
\usepackage[numbers]{natbib}
\usepackage{url}
\usepackage{enumitem}

\newtheorem{theorem}{Theorem}[section]
\newtheorem{lemma}[theorem]{Lemma}

\newtheorem*{thmA}{Theorem A}

\newtheorem*{thmB}{Theorem B}
\newtheorem*{thmC}{Theorem  C}
\newtheorem*{thmD}{Theorem  D}
\newtheorem*{thmE}{Theorem E}
\newtheorem*{CorF}{Corollary F}

\theoremstyle{remark}

\numberwithin{equation}{section}

\begin{document}
\title[Finite groups with simple or $p$-nilpotent maximal subgroups]
{Finite groups with simple or $p$-nilpotent maximal subgroups}
    \author[Beltr\'an and Viudez]{Antonio Beltr\'an\\
     Departamento de Matem\'aticas\\
      Universitat Jaume I \\
     12071 Castell\'on\\
      Spain\\
     \\Fernando Viudez \\
Departamento de Matem\'aticas\\
      Universitat Jaume I \\
     12071 Castell\'on\\
      Spain\\
     }

 \thanks{Antonio Beltr\'an: abeltran@uji.es ORCID ID: https://orcid.org/0000-0001-6570-201X \newline
 \indent Fernando Viudez: fviudez@uji.es }

\keywords{maximal subgroups, simple groups, $p$-nilpotent, $p$-decomposable}

\subjclass[2020]{20E28, 20D10, 20D06, 20D08}

\begin{abstract}
Let $p$ be a prime number. When $p$ is odd, we study finite groups in which every maximal subgroup is either non-abelian simple or $p$-nilpotent, as well as those in which
every  maximal subgroup is either non-abelian simple or $p$-decomposable. 
We prove that every non-simple, non-solvable group satisfying the first condition is $p$-nilpotent, and every non-simple, non-solvable group satisfying the second condition is $p$-decomposable. In addition, we determine the possible occurrence of non-abelian simple maximal subgroups in these cases. These results provide a substantial partial answer to two questions posed by V.S. Monakhov and I.N. Tyutyanov in the Kourovka Notebook concerning the non-abelian composition factors of such groups. For non-abelian simple groups, we determine those satisfying the corresponding conditions within the alternating and sporadic families. The case of simple groups of Lie type is left open. Finally, for $p=2$, we obtain a complete classification of the non-solvable finite groups whose maximal subgroups are either non-abelian simple or $2$-nilpotent.

\end{abstract}

\maketitle

\section{Introduction}
    Maximal subgroups play a decisive role in determining the structure of finite groups. Imposing strong conditions on these subgroups often  yields severe structural constraints on the whole group.  Classical precedents include the classification of minimal non-nilpotent groups (the so-called  Schmidt groups) and minimal non-$p$-nilpotent groups, which coincide with the ($p$-closed) Schmidt groups [\citenum{Huppert}, IV.5.4].   Further examples arise from the theory of minimal non-supersolvable groups, whose structure was investigated by Huppert and Doerk.
    
\medskip
In this paper, we study finite groups whose maximal subgroups are either non-abelian simple or $p$-nilpotent, and a related class in which $p$-nilpotent is replaced by $p$-decomposable. Our main results give a complete structural description of the non-simple, non-solvable groups in these classes. We also investigate the simple groups satisfying the same conditions in the alternating and sporadic cases, and obtain a complete classification in each family. 

    \medskip
In 2014, Monakhov and Tyutyanov initiated the study of finite groups whose maximal subgroups are non-abelian simple or $2$-nilpotent. More precisely, [\citenum{Monakhov}, Theorem 2] shows that such a group has at most one non-abelian composition factor. In fact, if $G$ is such a group, then it has a normal subgroup $K\cong {\rm PSL}_2(p^n)$, where $p$ is a prime and $n\geq 1$,  such that $|G:K|=1$ or a prime. The same paper also considers groups whose maximal subgroups are simple or nilpotent. Such groups are solvable  according to [\citenum{Monakhov}, Theorem 1]; more precisely, they are Schmidt groups. The authors also determined the structure of those groups in which every maximal subgroup is non-abelian simple or supersolvable [\citenum{Monakhov}, Corollary]. 
These results motivated two questions concerning odd primes posed by Monakhov and Tyutyanov in the Kourovka Notebook [\citenum{Kourovka}, Problems 19.57 and 19.58].  
\medskip
 
{\bf Question 1}. What are the non-abelian composition factors of a finite group in which every
maximal subgroup is simple or $p$-nilpotent for some fixed odd prime $p \in \pi(G)$?

\medskip
{\bf Question 2}. What are the non-abelian composition factors of a finite group in which every
maximal subgroup is simple or $p$-decomposable for some fixed odd prime $p \in \pi(G)$?

\medskip

The results obtained in this paper give a substantial partial answer to these questions. In particular, we completely determine the structure of the non-simple, non-solvable groups satisfying the corresponding hypotheses.
 For convenience, we say that a group $G$ satisfies  $(*_p)$ (respectively, ($*_{p-p'}$)) when every maximal subgroup of $G$ is either non-abelian simple or $p$-nilpotent (respectively, $p$-decomposable). 

\begin{thmA} Let $G$ be a finite non-simple, non-solvable group, and let $p\in \pi (G)$ be an odd prime.  Then $G$ satisfies $(*_p)$ if and only if $G$ is $p$-nilpotent.  Moreover, in that case, $G$ has at least one non-abelian simple maximal subgroup $M$ if and only if $G\cong  M\times C_p$ or $G\cong M\rtimes C_p$, where $M$ is a non-abelian simple $p'$-group. In either case, the maximal subgroup $M$ is unique.  
\end{thmA}

Theorem A provides a complete description of the non-simple, non-solvable groups that satisfy $(*_p)$. In particular, the only possible non-abelian composition factors of such a group are non-abelian simple $p'$-groups. Conversely, every non-abelian simple $p'$-group can occur as a composition factor of a finite group satisfying $(*_p)$. Furthermore, applying Theorem A yields the following result.

\begin{thmB} Let $G$ be a finite non-simple, non-solvable group, and  let $p\in \pi (G)$ be an odd prime.  Then $G$ satisfies $(*_{p-p'})$ if and only if $G$ is $p$-decomposable.  Furthermore,  in this case, $G$ possesses at least one non-abelian simple maximal subgroup if and only if $G\cong  M\times C_p$ where $M$ is a non-abelian simple $p'$-group. In this situation, such a maximal  subgroup is unique.  
\end{thmB}

Moreover, Theorem B shows that the non-abelian composition factors of such groups are non-abelian simple \(p'\)-groups. Conversely, every non-abelian simple \(p'\)-group may occur as a composition factor of a finite group satisfying $(*_{p-p'})$.

\medskip

The remaining  case in Questions 1 and 2 is that in which the group itself is non-abelian simple. We determine the simple groups satisfying $(*_p)$ or $(*_{p-p'})$  within the alternating and sporadic simple groups.

\begin{thmC}
Let $A_n$ be the alternating group, where  $n \ge 5$, and let  $p \in \pi(A_n)$.  The conditions $(\ast_p)$ and $(\ast_{p-p'})$ are equivalent for $A_n$. Furthermore, $A_n$ satisfies either (and hence both) of these properties if and only if:
\begin{itemize}
    \item[\normalfont(a)] $p = n - 1$  and $n$ is not a prime power; or
    \item[\normalfont(b)] $p = n$ and $n \in \{7, 11, 23\}$.
\end{itemize}
\end{thmC}

\begin{thmD}
Let \(G\) be a sporadic simple group and let \(p\in\pi(G)\). Then \(G\) satisfies \((\ast_p)\) if and only if one of the following cases occurs:
\begin{enumerate}[label=\normalfont(\alph*)]
    \item \(p=11\) and \(G\cong M_{11}, M_{12}, M_{22}, M_{23}, HS, McL, O'N\);
    \item \(p=13\) and \(G\cong \mathrm{Fi}_{22}\);
    \item \(p=19\) and \(G\cong J_3\);
    \item \(p=23\) and \(G\cong M_{24}, \mathrm{Co}_3, \mathrm{Co}_2\);
    \item \(p=29\) and \(G\cong Ru\);
    \item \(p=31\) and \(G\cong O'N\);
    \item \(p=71\) and \(G\cong M\).
\end{enumerate}
Moreover, \(G\) satisfies \((\ast_{p-p'})\) if and only if one of the cases above occurs, except for \(G\cong M_{23}\) and \(p=11\).
\end{thmD}

The case of simple groups of Lie type is not considered here and remains open. Thus, our results provide a partial but substantial answer to Questions 1 and 2.

\medskip

We finally turn to the case $p=2$. As mentioned above, this case was previously studied by Monakhov and Tyutyanov \cite{Monakhov}. Here we refine their main result by completely classifying the non-solvable finite groups satisfying  $(*_2)$. Our proof relies, among other things, on the main result of \cite{Shao-Beltran}. 

\begin{thmE}
Let $G$ be a finite non-solvable group. Then $G$ satisfies $(\ast_2)$ if and only if $G$ is isomorphic to one of the following groups:
\begin{itemize}
\item[\normalfont(1)] 
\(
\mathrm{PSL}_2(q),
\)
where \(q=p^f\geq 5\) is odd and one of the following holds:
\begin{itemize}
\item[\normalfont(a)] \(f=1\) and \(p\equiv \pm 11,\pm19 \pmod {40}\);
\item[\normalfont(b)] \(f>1\) is odd and \(q\neq 3^\ell\) for every odd prime \(\ell\).
\end{itemize}
\vspace{0,3 cm}
\item[\normalfont(2)] 
\(
\mathrm{PGL}_2(q),
\)
where $q=p^{2^a}$, $p$ is an odd prime and $a\geq 0$, with $p\equiv \pm1\pmod 8$ when $a=0$;
\vspace{0,3 cm}
\item[\normalfont(3)] $G$ is isomorphic to the following subgroup of $\mathrm{P\Gamma L}_2(q)$
\[
\langle \mathrm{PSL}_2(q),\delta\varphi^{2^{a-1}}\rangle,
\]
where  $q=p^{2^a}$,  $a\geq 1$, with $p$ being an odd prime, and where $\delta$ and  $\varphi$ denote the diagonal and  the field automorphism induced by the Frobenius automorphism of  \(\mathrm{PSL}_2(q)\), respectively.
\end{itemize}
\end{thmE}

As an immediate consequence, we obtain the following result for the \(2\)-decomposable analogue.

\begin{CorF} Every finite group satisfying $(*_{2-2'})$ is solvable.
\end{CorF}

The paper is organized as follows. Section 2 contains the proofs of Theorems A and B. In Section 3, we develop preliminary results concerning permutation groups, affine groups and projective groups, and then apply them to the alternating and sporadic simple groups, proving Theorems C and D. Section 4 is devoted to the case $p=2$, where we prove Theorem E and derive Corollary F.

\section{Proofs of Theorems A and B}

The proofs of Theorems A and B do not require the Classification of Finite Simple Groups. Instead, we appeal to the following result  [\citenum{BeltranShao2023Invariant},  Corollary 4.2]: A finite group in which every maximal subgroup is either normal or $p$-nilpotent is $p$-solvable, provided that $p$ is an odd prime. Its proof is based on the Glauberman-Thompson criterion for $p$-nilpotence.

\begin{proof}[Proof of Theorem A]
If $G$ is $p$-nilpotent, then  $G$ trivially satisfies $(*_p)$. Conversely, we assume  that $G$ satisfies $(*_p)$. If every maximal subgroup of $G$ is $p$-nilpotent, then $G$ itself must be $p$-nilpotent; otherwise $G$ would be minimal non-$p$-nilpotent and hence solvable by It\^o's theorem [\citenum{Huppert}, Theorem IV.5.4], a contradiction. Henceforth, we assume that $G$ possesses a maximal subgroup $M$ that is non-abelian simple. Let $N$ be a proper minimal normal subgroup of $G$, which exists because $G$ is not simple. We  distinguish two cases: whether $N$ is contained in $M$ or not.

\medskip
Case 1. Suppose first  that $N\nleq M$. Since $M$ is simple, we have $N\cap M=1$, and the maximality of $M$ implies that $G=NM$. Hence $G/N\cong M$. Let $L/N$ be a maximal subgroup of $G/N$. Then $L$ is maximal in $G$ and properly contains $N$, so $L$ cannot be non-abelian simple. Thus, by hypothesis, $L$ is $p$-nilpotent. This establishes that every maximal subgroup of $M\cong G/N$ is $p$-nilpotent. Now, if $M$ were not $p$-nilpotent, it would be minimal non-$p$-nilpotent, and again solvable by It\^o's theorem, a contradiction. It follows that $M$ is $p$-nilpotent, and as $M$ is non-abelian simple, we deduce that $p\nmid |M|$.

Let $M_1$ be a maximal subgroup of $G$ containing $N$. If $M_1$ were non-abelian simple, then $N=M_1$, and consequently $G/N$ would have prime order, contradicting $G/N\cong M$. This shows that $M_1$ is $p$-nilpotent. Now, since $G=NM$ and $p\nmid |M|$, it follows that $p$ divides $|N|$. Accordingly, the subgroup $N\leq M_1$ is also $p$-nilpotent. But $N$ is a minimal normal subgroup of $G$,  so it is either elementary abelian or a direct product of isomorphic non-abelian simple groups. However, the fact that \(p\) divides \(|N|\) forces \(N\) to be an elementary abelian \(p\)-group.

Let $q\neq p$ be any prime divisor of $|M|$ and let $Q\in {\rm Syl}_q(M)$. The proper subgroup $NQ$ is contained in some maximal subgroup $H$ of $G$, which cannot be a non-abelian simple group because $N\unlhd H$; so, by hypothesis, it is $p$-nilpotent. In particular, $NQ$ is $p$-nilpotent as well, forcing $NQ= N \times Q$. Since this holds for every Sylow subgroup of $M$, we have $[N,M]=1$,  and then $G=M\times N$.  Because $M$ is maximal in $G$, it follows that  $|G:M|=|N|$ is prime, so $N\cong C_p$. We conclude that $G\cong M\times C_p$, with $M$ being a non-abelian simple $p'$-group.

\medskip
Case 2. Suppose next that $N\leq M$. The fact that $M$ is a non-abelian simple group gives $N=M$,  implying that $G/M\cong C_q$ for some prime $q$ (possibly equal to $p$). Let $L\neq M$ be a maximal subgroup of $G$. If $L$ were non-abelian simple, then $L\cap M\unlhd L$, so $L\cap M=1$ or $L\leq M$. The latter case is impossible, whereas the former yields $L\cong LM/M=G/M\cong C_q$, a contradiction in either case. We deduce that every maximal subgroup of $G$ distinct from $M$ is $p$-nilpotent. Consequently, each such subgroup is either normal or \(p\)-nilpotent. Since $p$ is odd,  applying [\citenum{BeltranShao2023Invariant}, Corollary 4.2]  allows us to deduce that \(G\) is \(p\)-solvable. Now, $M$ is a non-abelian simple composition factor of $G$, so $p\nmid |M|$. But, as $p$ divides $|G|$ and $G/M\cong C_q$, we necessarily obtain $q=p$. This shows
that $G\cong M\rtimes C_p$, where $M$ is a non-abelian simple $p'$-group.

\medskip
In both cases, we have shown that $G$ is $p$-nilpotent, thereby establishing the first equivalence in the statement.  The above paragraphs also show the asserted structure for $G$ when there exists  a non-abelian simple maximal subgroup  in $G$. Conversely,  if $G\cong M\times C_p$ or $G\cong M\rtimes C_p$, where $M$ is a non-abelian simple $p'$-group, then $M$ is clearly a  non-abelian simple maximal subgroup of $G$, and it is unique because, in both cases, $M$ is the normal $p'$-complement of $G$. The proof is complete.
\end{proof}

\begin{proof}[Proof of Theorem B]
If \(G\) is \(p\)-decomposable, it is evident that \(G\) satisfies \((\ast_{p-p'})\).
Conversely, assume that \(G\) satisfies \((\ast_{p-p'})\). Since every \(p\)-decomposable group is \(p\)-nilpotent, \(G\) also satisfies \((\ast_p)\). By Theorem A,  \(G\) is \(p\)-nilpotent. 
Moreover, we claim that every maximal subgroup of \(G\) is \(p\)-decomposable. Indeed, if \(M\) is a maximal subgroup of \(G\) that is non-abelian simple, then, given that $M$ is $p$-nilpotent,  \(M\) must be a $p'$-group; hence, it is trivially \(p\)-decomposable.

\medskip
We now prove that \(G\) itself is \(p\)-decomposable. As \(G\) is \(p\)-nilpotent, we  write \(G=N\rtimes P\), where \(N=O_{p'}(G)\) and \(P\in \operatorname{Syl}_p(G)\). Since \(G\) is non-solvable and \(G/N\) is a \(p\)-group, the subgroup \(N\) must be non-solvable. In particular, \(N\) is not a \(q\)-group for any prime \(q\in\pi(N)\). 
Let \(q\in\pi(N)\), and let \(Q\) be a \(P\)-invariant Sylow \(q\)-subgroup of \(N\), whose existence follows from  standard properties of coprime action (see, for instance, [\citenum{KurzweilStellmacher}, Chapter 8]). As \(Q<N\), we have \(QP<G\), so we can choose a maximal subgroup \(H\) of \(G\) that contains \(QP\). By the preceding paragraph, \(H\) is \(p\)-decomposable, so \(H=O_p(H)\times O_{p'}(H)\). Since \(P\in \operatorname{Syl}_p(H)\) and \(Q\) is a \(p'\)-subgroup of \(H\), we have \(P=O_p(H)\) and \(Q\leq O_{p'}(H)\). In particular,  \([Q,P]=1\). This holds for every \(P\)-invariant Sylow \(q\)-subgroup of \(N\) and for every \(q\in\pi(N)\), yielding \([N,P]=1\). Consequently, \(G=N\times P\), which means that \(G\) is \(p\)-decomposable, as desired. Thus, the first equivalence is proved.

\medskip

Suppose that \(G\) possesses a maximal subgroup \(M\) that is non-abelian simple. Since \(G\) satisfies \((\ast_p)\), Theorem A gives
\(
G\cong M\times C_p
\)
or
\(
G\cong M\rtimes C_p,
\)
where \(M\) is a non-abelian simple \(p'\)-group. As we have proved that \(G\) is \(p\)-decomposable, it follows that \(G\cong M\times C_p\). 
The converse is immediate, and such a maximal subgroup \(M\) is unique because it is the normal \(p'\)-complement of \(G\). The proof is complete.
\end{proof}

\section{Proofs of Theorems C and D}

The proof of Theorem C relies mainly on the classification of maximal subgroups of the alternating groups due to Liebeck, Praeger and Saxl \cite{LPS}, together with the following result on primitive permutation groups containing a cycle with a certain number of fixed points.

\begin{theorem}[{\cite[Theorem 1.2]{Jones}}]\label{thm:jones}
Let $G$ be a primitive permutation group of finite degree $n$, not containing the alternating group $A_n$. Suppose that $G$ contains a cycle fixing $k$ points, where $0 \le k \le n-2$. Then one of the following holds:
\begin{enumerate}[label=\normalfont(\arabic*)]
    \item $k=0$ and either:
    \begin{enumerate}[label=\normalfont(\alph*)]
        \item $C_p \le G \le \mathrm{AGL}_1(p)$ with $n=p$ prime, or
        \item $\mathrm{PGL}_d(q) \le G \le \mathrm{P\Gamma L}_d(q)$ with $n=(q^d-1)/(q-1)$ and $d\ge 2$ for some prime power $q$, or
        \item $G=\mathrm{PSL}_2(11)$, $M_{11}$ or $M_{23}$ with $n=11$, $11$ or $23$ respectively.
    \end{enumerate}
    
    \item $k=1$ and either:
    \begin{enumerate}[label=\normalfont(\alph*)]
        \item $\mathrm{AGL}_d(q) \le G \le \mathrm{A\Gamma L}_d(q)$ with $n=q^d$ and $d\ge 1$ for some prime power $q$, or
        \item $G=\mathrm{PSL}_2(p)$ or $\mathrm{PGL}_2(p)$ with $n=p+1$ for some prime $p\ge 5$, or
        \item $G=M_{11}$, $M_{12}$ or $M_{24}$ with $n=12$, $12$ or $24$ respectively.
    \end{enumerate}
    
    \item $k=2$ and $\mathrm{PGL}_2(q) \le G \le \mathrm{P\Gamma L}_2(q)$ with $n=q+1$ for some prime power $q$.
    
\end{enumerate}
\end{theorem}
We begin with two elementary observations on permutation groups that will be needed in the proof of Theorem C.

\begin{lemma}\label{lem:intransitive}
Let $G$ be a maximal subgroup of $A_n$, with $n=p+1$ and $p$ an odd prime. Suppose that $G$ is not transitive and contains a $p$-cycle. Then $G$ is the stabilizer of one point in $A_n$. In particular, $G\cong A_p$.
\end{lemma}

\begin{proof}
Let $\Omega=\{1,\ldots,p+1\}$. A $p$-cycle contained in $G$  moves exactly $p$ points and fixes exactly one point. These $p$ moved points lie in the same orbit of $G$. The fact that $G$ is not transitive implies that the remaining point must belong to an orbit of length $1$. Hence $G$ is contained in the stabilizer of this point in $A_n$, which is isomorphic to $A_p$. By maximality of $G$ and $A_p$ in $A_{p+1}$, the equality follows. 
\end{proof}

\begin{lemma}\label{lem:primitive}
Let $G$ be a transitive subgroup of $A_n$, with $n=p+1$ and $p$ an odd prime. Suppose that $G$ contains a $p$-cycle. Then $G$ is primitive.
\end{lemma}

\begin{proof}
Let $\Omega=\{1,\ldots,p+1\}$ and assume without loss of generality that $\sigma=(1,\ldots,p)\in G$. Suppose that $G$ is imprimitive, and let $\mathcal{B}$ be a non-trivial system of blocks. Let $B\in\mathcal{B}$ be the block that contains the fixed point of $\sigma$. Since $\sigma$ fixes this point, we have $B^\sigma=B$. Hence $B$ must have size $1$ or $p+1$, which contradicts the assumption that  $\mathcal{B}$ is a non-trivial system. We conclude that $G$ is primitive.
\end{proof}

We also need two elementary facts concerning the affine and projective linear actions. Although they are well known, we include their proofs for the reader's convenience.

\begin{lemma}\label{lem:affine-intersection}
Let \(p\) be an odd prime. Under
the natural action of  \(\mathrm{AGL}_1(p)\) on \(\mathbb F_p\), we have
\[
\mathrm{AGL}_1(p)\cap A_p\cong C_p\rtimes C_{(p-1)/2}.
\]
\end{lemma}

\begin{proof}
We identify \(\mathrm{AGL}_1(p)\) with the group of affine maps
\[
x\mapsto ax+b,\qquad a\in\mathbb F_p^\times,\ b\in\mathbb F_p.
\]
Recall that \(\mathrm{AGL}_1(p)=T\rtimes M\), where
\[
T=\{x\mapsto x+b:b\in\mathbb F_p\}\cong C_p
\quad\text{and}\quad
M=\{x\mapsto ax:a\in\mathbb F_p^\times\}\cong C_{p-1}.
\]
Every non-trivial element of \(T\) acts as a \(p\)-cycle, and hence is an even permutation, since \(p\) is odd. Thus, \(T\leq A_p\). 
Now, if \(a=u^2\) is a square in \(\mathbb F_p^\times\), then the permutation \(x\mapsto ax\) is the square of the permutation \(x\mapsto ux\), and, as a consequence, it is even. Therefore, the subgroup of squares of \(\mathbb F_p^\times\), which has order \((p-1)/2\), is contained in \(M\cap A_p\). On the other hand, if \(g\) is a generator of \(\mathbb F_p^\times\), then the permutation \(x\mapsto gx\) fixes \(0\) and acts on \(\mathbb F_p^\times\) as the cycle
\(
(1,\ g,\ g^2,\ \cdots\ ,g^{p-2}),
\)
which has length \(p-1\). Since \(p-1\) is even, this is an odd permutation. We deduce that \(M\) is not contained in \(A_p\). Moreover, since the subgroup of squares has index \(2\) in \(M\), it follows that \(M\cap A_p\cong C_{(p-1)/2}\). We conclude that
\[
\mathrm{AGL}_1(p)\cap A_p
= T\rtimes (M\cap A_p)
\cong C_p\rtimes C_{(p-1)/2}.
\]
\end{proof}

Before stating the next lemma,  recall the natural action of \(\mathrm{PGL}_2(q)\) on the projective line. The projective line \(\mathbb P^1(\mathbb F_q)\) is the set of one-dimensional subspaces of \(\mathbb F_q^2\), consisting of \(q+1\) points. Each non-zero vector \((x,y)\in\mathbb F_q^2\) determines the point \(\langle(x,y)\rangle=\{\lambda(x,y):\lambda \in \mathbb{F}_q \}\). If \(y\neq0\), this point is represented by \(x/y\in\mathbb F_q\), whereas the pairs with \(y=0\) correspond to the point at infinity, denoted by \(\infty\). Thus, we identify
\[
\mathbb P^1(\mathbb F_q)=\mathbb F_q\cup\{\infty\}.
\]
Under this identification, an element of \(\mathrm{PGL}_2(q)\), represented by a matrix
\[
\begin{pmatrix}
a & b\\
c & d
\end{pmatrix}\in\mathrm{GL}_2(q)
\]
acts by the fractional linear transformation:
\[
z\longmapsto \frac{az+b}{cz+d}.
\]
Here, the usual conventions are used: if \(cz+d=0\), then the image of \(z\) is \(\infty\), and
\[
\infty\longmapsto
\begin{cases}
a/c, & \text{if } c\neq 0,\\
\infty, & \text{if } c=0.
\end{cases}
\]
Therefore, we may write
\[
\mathrm{PGL}_2(q)=
\left\{
z\mapsto \frac{az+b}{cz+d}\ :\ a,b,c,d\in\mathbb F_q,\ ad-bc\neq 0
\right\},
\]
where two matrices differing by a non-zero scalar induce the same transformation.

\begin{lemma}\label{lem:pgl2-not-in-alt}
Let \(p\) be an odd prime. Under its natural action on the projective line
\(\mathbb P^1(\mathbb F_p)\), the group \(\mathrm{PGL}_2(p)\) is not contained in
\(A_{p+1}\).
\end{lemma}

\begin{proof}
We identify \(\mathbb P^1(\mathbb F_p)\) with \(\mathbb F_p\cup\{\infty\}\). Let \(a\) be a generator of the cyclic group \(\mathbb F_p^\times\). The projective transformation \(x\mapsto ax\) belongs to \(\mathrm{PGL}_2(p)\), since it is induced by the matrix
\[
\begin{pmatrix}
a & 0\\
0 & 1
\end{pmatrix}.
\]
It fixes \(0\) and \(\infty\), and acts as a single cycle of length \(p-1\) on \(\mathbb F_p^\times\). Then, as a permutation of the \(p+1\) points of \(\mathbb P^1(\mathbb F_p)\), it has sign
\(
(-1)^{p-2}=-1
\)
because \(p-1\) is even. Then \(\mathrm{PGL}_2(p)\) contains an odd permutation and, thereby, \(\mathrm{PGL}_2(p)\nleq A_{p+1}\).
\end{proof}

We are ready to prove Theorem C.

\begin{proof}[Proof of Theorem C]
Since every \(p\)-decomposable group is \(p\)-nilpotent, property \((\ast_{p-p'})\) trivially entails \((\ast_p)\). Thus, to prove the theorem, we first determine when \(A_n\) satisfies \((\ast_p)\), and then prove  that, in these cases, every maximal subgroup of \(A_n\) is either non-abelian simple or has order not divisible by \(p\). As a consequence, \(A_n\) will satisfy \((\ast_{p-p'})\) as well. Therefore, for alternating groups, both conditions will be equivalent in the cases we list below. We divide the proof into three cases:

\medskip
Case 1. We show that $A_n$ does not satisfy the property $(\ast_p)$ whenever $p\leq n-2$. Assume first that $n\geq 7$ and consider the subgroup
\[
H=(S_2\times S_{n-2})\cap A_n.
\]
This is a maximal intransitive subgroup of $A_n$ by \cite[Theorem, case (a)]{LPS}, and furthermore, it is well known that $H\cong S_{n-2}$. This subgroup is clearly neither non-abelian simple nor \(p\)-nilpotent for any prime $p\leq n-2$.  Hence, $A_n$ does not satisfy the property $(\ast_p)$.

In this case, it remains to consider the values $n=5,6$. If $n=5$, then $p=2$ or $p=3$. When $p=2$, the maximal subgroup $A_4$ of $A_5$ is neither simple nor $2$-nilpotent. The same occurs with $p=3$ and the maximal subgroup $S_3$ of $A_5$. If $n=6$, then again  $p=2$ or $p=3$, and the maximal subgroup $(S_2\times S_4)\cap A_6\cong S_4$ is neither simple nor $p$-nilpotent.  As a result, we conclude that $A_n$ does not satisfy the property $(\ast_p)$ in these cases either.

\medskip
Case 2. We consider the case $p=n$. Let $H$ be the subgroup
\[
H=\mathrm{AGL}_1(p)\cap A_p.
\]
By Lemma~\ref{lem:affine-intersection}, we have \(H\cong C_p\rtimes C_{(p-1)/2}\) and by \cite[Theorem, case (c), and Table I]{LPS}, this affine subgroup is maximal in \(A_p\), except when \(p\in\{7,11,17,23\}\). If \(p\notin\{7,11,17,23\}\), then \(H\) is neither non-abelian simple nor \(p\)-nilpotent. Therefore, \(A_p\) does not satisfy the property \((\ast_p)\).

We are left with the exceptional primes $p=7,11,17,23$. For $p=7$, it is known from \cite{Atlas} that the maximal subgroups of $A_7$ whose orders are divisible by $7$ belong to two classes of subgroups isomorphic to the simple group $\mathrm{PSL}_2(7)$. The remaining maximal subgroups are $7'$-groups, which are trivially $7$-nilpotent. We obtain that $A_7$ satisfies property $(\ast_7)$.

With regard to $p=11$,  for example,  \cite{Atlas} shows that the only maximal subgroups of $A_{11}$ whose orders are divisible by $11$ belong to two classes of subgroups, which are isomorphic to  the simple group $M_{11}$. The remaining maximal subgroups are $11'$-groups, whence they are $11$-nilpotent. Therefore, $A_{11}$ satisfies the property $(\ast_{11})$.

For \(p=17\), it is not difficult to see that  \(\mathrm{PSL}_2(16):C_4\) is a maximal subgroup of \(A_{17}\). This can be checked, for instance, with \cite{GAP4} since this is one of the primitive subgroups of \(A_{17}\) of degree \(17\). This subgroup is neither simple nor \(17\)-nilpotent. As a result, \(A_{17}\) does not satisfy property \((\ast_{17})\).

Finally, for $p=23$, let $M$ be any maximal subgroup of $A_{23}$ of order divisible by $23$. Since $M$ contains a $23$-cycle,  $M$ is transitive and, furthermore, since its degree is prime, $M$ is primitive. We can then apply Theorem \ref{thm:jones}(1) and distinguish among three possibilities: the affine case (a), the projective case (b), and the Mathieu group  $M_{23}$ (c). The projective case cannot occur, since there are no prime powers $q$ and integers $d\geq 2$ satisfying
\[
23 = \frac{q^d-1}{q-1}.
\]
 On the other hand, case (a) forces \(M\cong \mathrm{AGL}_1(23)\cap A_{23}\cong C_{23}\rtimes C_{11}\), according to Lemma~\ref{lem:affine-intersection}, but this subgroup is not maximal in \(A_{23}\), because the exceptional containment in \cite[Table I]{LPS} asserts \(C_{23}\rtimes C_{11}<M_{23}<A_{23}\). In view of this,  the only maximal subgroups of $A_{23}$ whose order is divisible by $23$ are isomorphic to $M_{23}$. Consequently, $A_{23}$ satisfies property $(\ast_{23})$.

\medskip
Case 3. Finally, we study the case $p=n-1$.  Let $M$ be a maximal subgroup of $A_{p+1}$. If $p\nmid |M|$, then $M$ is trivially $p$-nilpotent. Henceforth, we assume that $p$ divides $|M|$, and choose an element of order $p$, which is necessarily a $p$-cycle. Suppose first that $M$ is intransitive. Then Lemma \ref{lem:intransitive} implies that $M\cong A_p$, which is simple and maximal in $A_{p+1} $ for $p\geq 5$. Assume then that $M$ is transitive. By Lemma \ref{lem:primitive}, $M$ is primitive. Applying Theorem \ref{thm:jones}(2) we see that the possible cases are: affine groups; $\mathrm{PSL}_2(p)$ or $\mathrm{PGL}_2(p)$ under their natural action of degree $p+1$; and Mathieu groups in exceptional degrees. We analyze this situation by distinguishing two cases depending on whether $n=p+1$ is a prime power.

If \(p+1\) is not a prime power, the affine case (a) does not occur. On the other hand, by Lemma~\ref{lem:pgl2-not-in-alt}, \(\mathrm{PGL}_2(p)\) is not contained in \(A_{p+1}\), whereas \(\mathrm{PSL}_2(p)\) and the relevant Mathieu groups are non-abelian simple. This allows us to ensure that \(A_{p+1}\) satisfies the property \((\ast_p)\) when \(p=n-1\) and \(n\) is not a prime power.

Suppose finally that $p+1$ is a prime power. As $p$ is an odd prime,  $p+1$ is even, implying that $p+1=2^d$, that is, $p=2^d-1$. Moreover, since $n=p+1\geq 5$, we have $d\geq 3$. Then the affine group
\[
\mathrm{AGL}_d(2)=C_2^d\rtimes \mathrm{GL}_d(2)
\]
is a maximal subgroup of $A_{2^d}$ by \cite{LPS}. This subgroup is evidently not simple. We claim that it is not $p$-nilpotent either. Indeed, if it were $p$-nilpotent, then its quotient  $\mathrm{GL}_d(2)$ would also be $p$-nilpotent. However, $\mathrm{GL}_d(2)=\mathrm{PSL}_d(2)$ is non-abelian simple for $d\geq 3$, while  $p=2^d-1$ divides 
\[
|\mathrm{GL}_d(2)|=2^{d(d-1)/2}(2^d-1)(2^{d-1}-1)\cdots(2^2-1),
\]
leading to a contradiction. The proof is now complete.
\end{proof}

\begin{proof}[Proof of Theorem D]
To investigate the structure of the sporadic simple groups, we will mainly rely on \cite{Atlas}. However, for \(\mathrm{Fi}_{22}\), \(\mathrm{Fi}_{23}\), \(\mathrm{Fi}_{24}'\), \(J_4\), \(Th\), \(B\) and \(M\), where the information in \cite{Atlas} is incomplete or has been refined or updated, we consult  the subsequent corresponding classifications  given in \cite{KleidmanWilsonFi22,KleidmanWilsonFi22Corr,KleidmanFi23,LintonFi24,KleidmanJ4,LintonTh,WilsonB,DietrichLeePopielMonster}, adopting the notation of each respective source. 
We first discuss, from Cases 1 to 19, those sporadic groups whose maximal subgroups can be obtained from \cite{Atlas}. However, for \(\mathrm{Fi}_{22}\) (Case 17),  an additional reference is explicitly indicated.

\smallskip

\noindent\((1)\) \(G\cong M_{11}\). We observe that \(S_5\) is a maximal subgroup that fails to be non-abelian simple or \(p\)-nilpotent for \(p\in\{2,3,5\}\). For \(p=11\), the only maximal subgroups of \(M_{11}\) whose orders are divisible by \(11\) are isomorphic to \(L_2(11)\) which is non-abelian simple. Therefore, \(M_{11}\) satisfies both \((\ast_{11})\) and \((\ast_{11-11'})\), and this occurs only for \(p=11\).

\smallskip

\noindent\((2)\) \(G\cong M_{12}\). The maximal subgroup \(2\times S_5\) is neither non-abelian simple nor \(p\)-nilpotent for \(p\in\{2,3,5\}\). For \(p=11\), the only maximal subgroups whose orders are divisible by \(11\) are those isomorphic to \(M_{11}\) and \(L_2(11)\), both of which are non-abelian simple. This establishes that \(M_{12}\) satisfies both \((\ast_{11})\) and \((\ast_{11-11'})\), and this is true
only for \(p=11\).

\smallskip

\noindent\((3)\) \(G\cong J_1\). The maximal subgroup \(2\times A_5\) is neither non-abelian simple nor \(p\)-nilpotent for \(p\in\{2,3,5\}\). For $p= 7, 11, 19$, the maximal subgroups \(7:6\), \(11:10\) and \(19:6\), respectively, are not $p$-nilpotent nor simple. Thus, \(J_1\) satisfies neither \((\ast_p)\) nor \((\ast_{p-p'})\) for any prime \(p\in\pi(J_1)\).

\smallskip

\noindent\((4)\) \(G\cong M_{22}\). The maximal subgroup \(2^4:S_5\) is neither non-abelian simple nor \(p\)-nilpotent for \(p\in\{2,3,5\}\), while \(2^3:L_3(2)\) fails the condition for \(p=7\). There is only one class of maximal subgroups of order divisible by   \(11\), which corresponds to the simple group \(L_2(11)\). We find that \(M_{22}\) satisfies both \((\ast_{11})\) and \((\ast_{11-11'})\), a fact that occurs  only for \(p=11\).

\smallskip

\noindent\((5)\) \(G\cong J_2\). The maximal subgroup \(A_4\times A_5\) fails to be simple or \(p\)-nilpotent for \(p\in\{2,3,5\}\). For $p=7$,  the maximal subgroup \(L_3(2):2\) is neither non-abelian simple nor \(7\)-nilpotent. This means that \(J_2\) satisfies neither \((\ast_p)\) nor \((\ast_{p-p'})\) for any prime \(p\in\pi(J_2)\).

\smallskip

\noindent\((6)\) \(G\cong M_{23}\). The maximal subgroup \(2^4:A_7\) is neither non-abelian simple nor \(p\)-nilpotent for \(p\in\{2,3,5,7\}\). In addition, \(23:11\) violates the condition for \(p=23\). For \(p=11\), the maximal subgroups whose orders are divisible by \(11\) are isomorphic to \(M_{22}\), \(M_{11}\), and \(23:11\). Now, since the Mathieu groups are simple and \(23:11\) is \(11\)-nilpotent, it follows that \(M_{23}\) satisfies \((\ast_p)\) precisely for \(p=11\). However, because \(23:11\) is not \(11\)-decomposable, we can assert that \(M_{23}\) satisfies \((\ast_{p-p'})\) for no prime \(p\in\pi(M_{23})\).

\smallskip

\noindent\((7)\) \(G\cong HS\). The maximal subgroup \(5:4\times A_5\) is neither non-abelian simple nor \(p\)-nilpotent for \(p\in\{2,3,5\}\), while \(S_8\) is neither \(7\)-nilpotent nor simple. On the other hand,   \(M_{22}\) and \(M_{11}\) are the only (classes of)  maximal subgroups with order divisible by \(11\). We conclude that \(HS\) satisfies both \((\ast_{11})\) and \((\ast_{11-11'})\), and this is  unique to \(p=11\).

\smallskip

\noindent\((8)\) \(G\cong J_3\). The maximal subgroups isomorphic to  \(L_2(16):2\) are neither non-abelian simple nor \(p\)-nilpotent for \(p\in\{2,3,5,17\}\). For \(p=19\), the maximal subgroups whose order is divisible by \(19\) are exactly those lying in two classes of subgroups isomorphic to the simple group \(L_2(19)\). This implies that \(J_3\) satisfies both \((\ast_{19})\) and \((\ast_{19-19'})\), and no other prime $p$ has this property.

\smallskip

\noindent\((9)\) \(G\cong M_{24}\). The maximal subgroup \(M_{22}:2\)  fails to be non-abelian simple or \(p\)-nilpotent for \(p\in\{2,3,5,7,11\}\). Moreover, we find that the only maximal subgroups whose order is divisible by \(23\) are those isomorphic to \(M_{23}\) and \(L_2(23)\), both of which are  simple. This shows that \(M_{24}\) satisfies both \((\ast_{23})\) and \((\ast_{23-23'})\), a feature that is valid only when \(p=23\).

\smallskip

\noindent\((10)\) \(G\cong McL\). The maximal subgroup \(2^4:A_7\) is neither non-abelian simple nor \(p\)-nilpotent for \(p\in\{2,3,5,7\}\). Furthermore, the only maximal subgroups whose orders are divisible by \(11\) lie in two classes of subgroups isomorphic to \(M_{22}\) and \(M_{11}\), respectively, all being simple. Consequently, \(McL\) satisfies both \((\ast_{11})\) and \((\ast_{11-11'})\), which occurs only when \(p=11\).

\smallskip

\noindent\((11)\) \(G\cong He\). The maximal subgroup \(S_4(4):2\) lacks both conditions for \(p\in\{2,3,\allowbreak 5,17\}\), while  for \(p=7\), we may consider \(7:3 \times 
L_3(2)\). Hence \(He\) does not satisfy either \((\ast_p)\) or \((\ast_{p-p'})\) for any prime \(p\in\pi(He)\).

\smallskip

\noindent\((12)\) \(G\cong Ru\). The maximal subgroup \({}^2F_4(2)\) fails to be simple or \(p\)-nilpotent for \(p\in\{2,3,5,13\}\), whereas \(L_2(13):2\) is not simple nor \(7\)-nilpotent. Regarding \(p=29\), the only maximal subgroups with order  divisible by \(29\) are those isomorphic to \(L_2(29)\). This implies that \(Ru\) satisfies both \((\ast_{29})\) and \((\ast_{29-29'})\), a fact unique to  \(p=29\).

\smallskip

\noindent\((13)\) \(G\cong Suz\). The maximal subgroup \(J_2:2\) is neither non-abelian simple nor \(p\)-nilpotent for \(p\in\{2,3,5,7\}\). In addition, \(M_{12}:2\) fails both properties for \(p=11\), and  \(L_3(3):2\) fails similarly for \(p=13\). In view of this, \(Suz\) satisfies neither \((\ast_p)\) nor \((\ast_{p-p'})\) for any prime \(p\in\pi(Suz)\).

\smallskip

\noindent\((14)\) \(G\cong O'N\). The maximal subgroup \(L_3(7):2\) is neither non-abelian simple nor \(p\)-nilpotent for \(p\in\{2,3,7,19\}\), whereas  \((3^2:4\times A_6).2\) is neither \(5\)-nilpotent nor simple. For \(p=11\), the only maximal subgroups whose orders are divisible by \(11\) are isomorphic to \(J_1\) and \(M_{11}\), both simple. Similarly,  for \(p=31\), the only maximal subgroups whose orders are divisible by \(31\) are those that belong to two classes of subgroups isomorphic to \(L_2(31)\),  which are also simple. This means that \(O'N\) satisfies \((\ast_p)\) and \((\ast_{p-p'})\) precisely for \(p\in\{11,31\}\).

\smallskip

\noindent\((15)\) \(G\cong \mathrm{Co}_3\). The maximal subgroup \(McL:2\) is not simple or \(p\)-nilpotent for \(p\in\{2,3,5,7,11\}\). For \(p=23\), the only class of maximal subgroups having order  divisible by \(23\) is \(M_{23}\), which is non-abelian simple. Consequently,  \(\mathrm{Co}_3\) satisfies both \((\ast_{23})\) and \((\ast_{23-23'})\). This property occurs only when \(p=23\).

\smallskip

\noindent\((16)\) \(G\cong \mathrm{Co}_2\). The maximal subgroup \(U_6(2):2\) is neither non-abelian simple nor \(p\)-nilpotent for \(p\in\{2,3,5,7,11\}\). For \(p=23\), the only maximal subgroup with order divisible by \(23\) is \(M_{23}\), which is non-abelian simple. We find that \(\mathrm{Co}_2\) satisfies both \((\ast_{23})\) and \((\ast_{23-23'})\), and this property is unique to \(p=23\).

\smallskip

\noindent\((17)\) \(G\cong \mathrm{Fi}_{22}\). The list of maximal subgroups in \cite{Atlas} is incomplete, so we use the complete classification in \cite{KleidmanWilsonFi22,KleidmanWilsonFi22Corr}. The maximal subgroup \(2.U_6(2)\) is not simple or \(p\)-nilpotent for \(p\in\{2,3,5,7,11\}\). With regard to \(p=13\), the only maximal subgroups whose orders are divisible by \(13\) belong to two classes of subgroups isomorphic to \(O_7(3)\) and to the class of \({}^2F_4(2)'\), all of which are non-abelian simple. Consequently, \(\mathrm{Fi}_{22}\) satisfies both \((\ast_{13})\) and \((\ast_{13-13'})\), which occurs only for \(p=13\).

\smallskip

\noindent\((18)\) \(G\cong HN\). The maximal subgroup \(2 \cdot HS.2\) lacks both conditions for \(p\in\{2,3,\allowbreak 5,7,11\}\), while for \(p=19\) it is enough to consider  \(U_3(8):3\), which is not $19$-nilpotent. This shows that \(HN\) satisfies neither \((\ast_p)\) nor \((\ast_{p-p'})\) for any prime \(p\in\pi(HN)\).

\smallskip

\noindent\((19)\) \(G\cong Ly\). The maximal subgroup \(3 \cdot McL:2\) is neither non-abelian simple nor \(p\)-nilpotent for \(p\in\{2,3,5,7,11\}\), while \(5^3 \cdot L_3(5)\) is not \(31\)-nilpotent or non-abelian simple. On the other hand, \(37:18\) and \(67:22\) lack both properties for \(p=37\) and \(p=67\), respectively. Thus, \(Ly\) satisfies neither \((\ast_p)\) nor \((\ast_{p-p'})\) for any prime \(p\in\pi(Ly)\).

\smallskip

We now turn to the remaining sporadic groups whose maximal subgroups are not fully covered by \cite{Atlas}, relying instead on the corresponding complete classifications cited in the following.

\smallskip

\noindent\((20)\) \(G\cong Th\). We refer to the complete classification of its maximal subgroups given in \cite{LintonTh}. The maximal subgroup \(2^5 \cdot L_5(2)\) fails to be simple or \(p\)-nilpotent for \(p\in\{2,3,5,7,31\}\). In addition, \({}^3D_4(2):3\) is neither \(13\)-nilpotent nor simple, while \(U_3(8):6\) violates the condition for \(p=19\). We conclude that \(Th\) does not satisfy \((\ast_p)\) or \((\ast_{p-p'})\) for any prime \(p\in\pi(Th)\).

\smallskip

\noindent\((21)\) \(G\cong \mathrm{Fi}_{23}\). According to the classification of its maximal subgroups in \cite{KleidmanFi23}, the maximal subgroup \(2\cdot \mathrm{Fi}_{22}\) lacks the conditions for \(p\in\{2,3,5,7,11,13\}\). In addition, \(S_4(4):4\) is neither \(17\)-nilpotent nor simple, while \(2^{11} \cdot M_{23}\)  violates the condition for \(p=23\). Therefore, \(\mathrm{Fi}_{23}\) does not  satisfy  \((\ast_p)\) or \((\ast_{p-p'})\) for any prime \(p\in\pi(\mathrm{Fi}_{23})\).

\smallskip

\noindent\((22)\) \(G\cong \mathrm{Co}_1\). The maximal subgroups are already given in \cite{Atlas}, so no later classification is needed in this case. The maximal subgroup \(3 \cdot Suz : 2\) is neither simple nor \(p\)-nilpotent for \(p\in\{2,3,5,7,11,13\}\), while we may consider \(2^{11}:M_{24}\)  for  \(p=23\). It follows that \(\mathrm{Co}_1\) does not satisfy \((\ast_p)\) or \((\ast_{p-p'})\) for any prime \(p\in\pi(\mathrm{Co}_1)\).

\smallskip

\noindent\((23)\) \(G\cong J_4\). We refer to the complete classification of its maximal subgroups given in \cite{KleidmanJ4}. The maximal subgroup \(2^{11}:M_{24}\) fails to be simple or \(p\)-nilpotent for \(p\in\{2,3,5,7,11,23\}\). Also, the maximal subgroups \(L_2(32):5\), \(29:28\), \(37:12\), and \(43:14\) are not \(31\)-nilpotent, \(29\)-nilpotent, \(37\)-nilpotent, and \(43\)-nilpotent, respectively, nor are they simple. Consequently, \(J_4\) satisfies neither \((\ast_p)\) nor \((\ast_{p-p'})\) for any prime \(p\in\pi(J_4)\).

\smallskip

\noindent\((24)\) \(G\cong \mathrm{Fi}_{24}'\). According to the complete classification of its maximal subgroups in \cite{LintonFi24}, the maximal subgroup \(2 \cdot \mathrm{Fi}_{22}:2\) is neither non-abelian simple nor \(p\)-nilpotent for \(p\in\{2,3,5,7,11,13\}\). Moreover, \(He:2\), \(2^{11} \cdot M_{24}\), and \(29:14\) fail to be non-abelian simple or \(17\)-nilpotent, \(23\)-nilpotent, and \(29\)-nilpotent, respectively, and none of them is simple. Hence \(\mathrm{Fi}_{24}'\)  does not satisfy \((\ast_p)\) or \((\ast_{p-p'})\) for any prime \(p\in\pi(\mathrm{Fi}_{24}')\).

\smallskip

\noindent\((25)\) \(G\cong B\). We rely on the complete classification of its maximal subgroups in \cite{WilsonB}. The maximal subgroup \(2^{1+22} \cdot \mathrm{Co}_2\) lacks the conditions for \(p\in\{2,3,5,7,11,23\}\). In addition, \((2^2\times F_4(2)):2\) fails to be non-abelian simple or \(p\)-nilpotent for \(p\in\{13,17\}\), and \(HN:2\) fails to be non-abelian simple or \(19\)-nilpotent. Finally, \(5^3 \cdot L_3(5)\) and \(47:23\) are neither \(31\)-nilpotent and \(47\)-nilpotent, respectively, nor non-abelian simple groups. Therefore, \(B\) satisfies neither \((\ast_p)\) nor \((\ast_{p-p'})\) for any prime \(p\in\pi(B)\).

\smallskip

\noindent\((26)\) \(G\cong M\). We utilize the complete classification of its maximal subgroups presented  in \cite{DietrichLeePopielMonster}. The maximal subgroup \(2 \cdot B\) lacks both conditions for  \(p\in\{2,3,5,7,11,13,17,19, \allowbreak 23,31,47\}\). On the other hand, \(3 \cdot \mathrm{Fi}_{24}\), \(41:40\), and \(59:29\) also lack both properties for \(p=29\), \(p=41\) and \(p=59\), respectively. Finally, the only class of maximal subgroups of order divisible by \(71\) corresponds to the simple group \(\mathrm{PSL}_2(71)\). Thus, \(M\) satisfies both \((\ast_{71})\) and \((\ast_{71-71'})\), and this only applies to \(p=71\).

\smallskip

Gathering the positive cases obtained above yields exactly the list in the statement. Moreover, this same list also comprises the groups that satisfy \((\ast_{p-p'})\), except, as previously pointed out, for \(M_{23}\) with \(p=11\), since \(C_{23} : C_{11}\) is \(11\)-nilpotent but not \(11\)-decomposable. This completes the proof of the theorem.
\end{proof}

We now turn to the part of the problem that remains open. As noted in the Introduction,  unlike the case for alternating and sporadic groups, a complete classification  of the simple groups of Lie type satisfying $(*_p)$ or $(*_{p-p'})$  for an arbitrary odd prime $p$ remains open. The intricate structure of the maximal subgroups of finite simple groups of Lie type makes this problem substantially more difficult. For example, $L_2(11)$ satisfies condition $(*_5)$ but not $(*_{5-5'})$; the unitary groups $U_5(2)$ and $U_6(2)$ satisfy  both $(*_{11})$ and $(*_{11-11'})$; and $U_3(3)$, $L_3(4)$, $U_3(5)$ and $U_4(3)$ satisfy both $(*_7)$ and $(*_{7-7'})$. These examples follow from the maximal subgroup data in \cite{Atlas}.
\section{Proofs of Theorem E and Corollary F}

For the reader’s convenience, and since the proof of Theorem E relies heavily on it, we reproduce 
below the list from the classification of maximal subgroups of low-dimensional classical groups; see [\citenum{BHRD}, Tables 8.1 and 8.7].

\begin{theorem}\label{thm:psl2-maximals}
Let \(q=p^f>3\) be a prime power. The maximal subgroups of \(\mathrm{PSL}_2(q)\) are the following:
\begin{itemize}
\item[\normalfont(1)] dihedral groups of order \(q-1\), for \(q\geq 13\) odd, and of order \(2(q-1)\), for \(q\) even;
\item[\normalfont(2)] dihedral groups of order \(q+1\), for \(q\neq 7,9\) odd, and of order \(2(q+1)\), for \(q\) even;
\item[\normalfont(3)] a point stabilizer, of order \(q(q-1)/2\) for \(q\) odd, and of order \(q(q-1)\) for \(q\) even;
\item[\normalfont(4)] \(\mathrm{PSL}_2(q_0)\), where \(q=q_0^\ell\), with \(\ell\) an odd prime if \(q\) is odd, and with \(\ell\) prime and \(q_0>2\) if \(q\) is even;
\item[\normalfont(5)] \(\mathrm{PGL}_2(q_0)\), where \(q=q_0^2\) and \(q\) is odd;
\item[\normalfont(6)] \(S_4\), when \(q\equiv\pm1\pmod 8\), with \(q\) prime;
\item[\normalfont(7)] \(A_4\), when \(q\equiv \pm3,5,\pm13 \pmod{40}\), with \(q\) prime;
\item[\normalfont(8)] \(A_5\), when \(q\equiv\pm1\pmod {10}\), with either \(q\) prime, or \(q=p^2\) and \(p\equiv\pm3\pmod {10}\).
\end{itemize}
\end{theorem}

\begin{proof}[Proof of Theorem E]
Assume that \(G\) satisfies \((\ast_2)\). By [\citenum{Monakhov}, Theorem 2], we know that \(G\) has a chief series \(1\unlhd T\unlhd G\), where \(T\cong \mathrm{PSL}_2(q)\), with \(q=p^f\) for some prime \(p\) and some integer \(f\geq 1\), and either \(|G:T|=1\) or \(|G:T|\) is prime. We distinguish these two cases.

\medskip
Case 1. Assume first that \(|G:T|=1\), that is, \(G \cong \mathrm{PSL}_2(q)\) is simple with \(q=p^f\). In the following, we determine when this group satisfies \((\ast_2)\) using Theorem~\ref{thm:psl2-maximals}.

First, suppose that \(q\) is even, so \(p=2\). Then \(q\geq4\), because \(\mathrm{PSL}_2(2)\cong S_3\) is not simple. By Theorem~\ref{thm:psl2-maximals}, \(G\) has a maximal subgroup isomorphic to \(C_2^f\rtimes C_{q-1}\), which is the stabilizer of a point of the projective line. However, this subgroup is neither non-abelian simple nor \(2\)-nilpotent, so we can assume henceforth that \(q\) is odd, and so is \(p\).

Assume first that \(f=1\), so \(q=p\) is prime. Taking into account Theorem~\ref{thm:psl2-maximals}, the point stabilizer and the maximal dihedral subgroups are \(2\)-nilpotent, while \(A_5\), when it occurs, is non-abelian simple. Thus, the only possible obstructions in this  case are \(A_4\) and \(S_4\). The subgroup \(S_4\) arises as a maximal subgroup when \(p\equiv\pm1\pmod 8\), so we must have \(p\equiv\pm3\pmod 8\). Under this assumption, \(A_4\) is a maximal subgroup precisely when \(p\equiv \pm3,5,\pm13 \pmod{40}\), so we must additionally require that \(p\not\equiv\pm3,5,\pm13 \pmod{40}\). Therefore, in the prime case, \(G\) satisfies \((\ast_2)\) if and only if \(p\equiv\pm3\pmod 8\) and \(p\not\equiv\pm3,5,\pm13 \pmod{40}\). A straightforward verification shows that this is equivalent to saying  \(p\equiv \pm11,\pm19\pmod {40}\).

Suppose now that \(f>1\). If \(f\) is even, then \(q=q_0^2\), with \(q_0=p^{f/2}\), and Theorem~\ref{thm:psl2-maximals} provides a maximal subgroup isomorphic to \(\mathrm{PGL}_2(q_0)\). This subgroup is neither non-abelian simple nor \(2\)-nilpotent. Hence, we can assume \(f\) to be odd, and \(\mathrm{PGL}_2(q_0)\) no longer occurs as a maximal subgroup. In this case, the subgroups \(A_4\) and \(S_4\) are not maximal subgroups according to Theorem~\ref{thm:psl2-maximals}. Thus, the only obstruction remaining is a maximal subgroup \(\mathrm{PSL}_2(q_0)\) with \(q_0=3\), since \(\mathrm{PSL}_2(3)\cong A_4\). This happens precisely when \(q=3^\ell\), with \(\ell\) an odd prime, and thus, these cases must be excluded. Conversely, apart from these excluded cases, we have: the point stabilizer and the dihedral maximal subgroups, which are \(2\)-nilpotent; the maximal subgroups \(\mathrm{PSL}_2(q_0)\) satisfying \(q_0\neq3\), so they are non-abelian simple; and \(A_5\), when it occurs, is non-abelian simple as well. Accordingly, \(G\) satisfies \((\ast_2)\), which completes the proof of part (1) of the statement.

\medskip
Case 2. Assume that \(|G:T|=s\) is prime and write \(G/T\cong C_s\). We first claim that every maximal subgroup of \(G\) different from \(T\) is \(2\)-nilpotent. Let \(L\neq T\) be a maximal subgroup of \(G\). If \(L\) were non-abelian simple, then \(L\cap T\unlhd L\), so \(L\cap T=1\) or \(L\leq T\). Since the latter case is impossible, it follows that \(LT=G\) and \(L\cong G/T\cong C_s\), a contradiction. Thus, \(L\) is \(2\)-nilpotent, as claimed. Since \(T\) is normal in \(G\), we deduce that every maximal subgroup of \(G\) is normal or \(2\)-nilpotent. We can then apply [\citenum{Shao-Beltran}, Theorem A], and taking into account that \(T\) is the unique non-abelian chief factor of \(G\), we conclude that \(T\cong \mathrm{PSL}_2(p^{2^a})\), where \(p\) is odd and either \(a\geq1\), or \(a=0\) and \(p\equiv\pm1\pmod 8\).

Next, we show that \(C_G(T)=1\). Suppose, on the contrary, that \(C=C_G(T)>1\). The fact that \(T\) is non-abelian simple implies that \(C\cap T=1\). Moreover, since \(T\) is maximal in \(G\), we have \(G=TC=T\times C\), with \(C\cong C_s\). Now, not every maximal subgroup of \(T\) can be \(2\)-nilpotent; otherwise \(T\) would be minimal non-\(2\)-nilpotent and hence solvable by It\^o's theorem [\citenum{Huppert}, Theorem IV.5.4]. If we take \(U\) to be a maximal subgroup of \(T\) that is not \(2\)-nilpotent, then \(U\times C\) would be a maximal subgroup of \(G\) that is neither non-abelian simple nor \(2\)-nilpotent. This contradiction shows that \(C_G(T)=1\).

Therefore, \(G\leq \operatorname{Aut}(T)\) and \(G/T\leq \operatorname{Out}(T)\). Since \(T\cong \mathrm{PSL}_2(q)\), with \(q=p^{2^a}\) odd, it is well known that \(|\operatorname{Out}(T)|=(2,q-1)2^a=2^{a+1}\) (see, for instance, \cite{Atlas}). Hence, \(G/T\) has order $2$. We use the standard notation
\[
\mathrm{Aut}(T)=\mathrm{P\Gamma L}_2(q)=\langle T,\delta,\varphi\rangle,
\]
where $\delta$ and  $\varphi$ denote the diagonal automorphism and  the field automorphism induced by the Frobenius automorphism of  \(\mathrm{PSL}_2(q)\), respectively. We  denote by \(\widehat{\alpha}=\alpha T\) the image of \(\alpha\in \mathrm{Aut}(T)\) in \(\mathrm{Out}(T)\). Note that
\[
C_2\cong G/T\leq \mathrm{Out}(T)\cong \langle \widehat{\delta}\rangle\times \langle \widehat{\varphi}\rangle\cong C_2\times C_{2^a}.
\]
If \(a=0\), then there are  no non-trivial field automorphisms, and necessarily \(G=\langle T,\delta\rangle=\mathrm{PGL}_2(p)\). If \(a\geq1\), letting \(\theta=\varphi^{2^{a-1}}\), the involutions in \(\mathrm{Out}(T)\) are  exactly \(\widehat{\delta}\), \(\widehat{\theta}\) and \(\widehat{\delta\theta}\). Therefore, the only possibilities are: \(G=\langle T,\delta\rangle=\mathrm{PGL}_2(q)\); \(G=\langle T,\theta\rangle\); and \(G=\langle T,\delta\theta\rangle=\langle \mathrm{PSL}_2(q),\delta\varphi^{2^{a-1}}\rangle\).

Next, we discard the pure field case. Suppose that
\(G=\langle T,\theta\rangle\), and write \(q=q_0^2\), where
\(q_0=p^{2^{a-1}}\). Since \(q\geq 4\), we can apply
[\citenum{Giudici2007PSL2Maximal}, Proposition 3.1] (converse direction) and then there exists a  maximal subgroup $H$ of \(G\)  such that \(H=N_G(H_0)\), where
\(H_0\cong \mathrm{PGL}_2(q_0)\).   Moreover, by
[\citenum{Giudici2007PSL2Maximal}, Lemma 2.4], we have
\(H/H_0\cong G/T\cong C_2\). In particular, \(H\) is
not simple, and is not \(2\)-nilpotent
because \(H_0\cong \mathrm{PGL}_2(q_0)\) is not. Indeed,
\(\mathrm{PGL}_2(3)\cong S_4\) is not \(2\)-nilpotent, and when
\(q_0>3\), the subgroup
\(\mathrm{PSL}_2(q_0)\unlhd \mathrm{PGL}_2(q_0)\) is non-abelian simple.
This shows that the pure field case cannot happen. Hence \(G\) is one of the groups listed in (2) and (3), which we discuss next separately.

\medskip
Suppose first that \(G=\langle \mathrm{PSL}_2(q),\delta\varphi^{2^{a-1}}\rangle\), with \(a\geq1\). By [\citenum{Giudici2007PSL2Maximal}, Theorem 1.5], the maximal subgroups of \(G\) that do not contain \(T\) are:
\[
C_p^{\,2^a}\rtimes C_{q-1},\qquad
N_G(D_{q-1}),\qquad
N_G(D_{q+1}),
\]
and possibly \(N_G(\mathrm{PSL}_2(q_0))\), where \(q=q_0^\ell\) for some odd prime \(\ell\). However, \(q=p^{2^a}\), so the last case cannot occur. The first subgroup is certainly 2-nilpotent. 
Let \(H\) be one of the remaining maximal subgroups (dihedral normalizer)  and put \(H_0=H\cap T\). Then \(H_0\) is equal to
\( D_{q-1}\) or \(D_{q+1}
\)
by [\citenum{Giudici2007PSL2Maximal}, Lemma 2.4]. Furthermore, applying this lemma, we have \(H=N_G(H_0)\) and \(H/H_0\cong G/T\cong C_2\).

Suppose that \(H_0=D_{q-1}\). Following the notation of [\citenum{Giudici2007PSL2Maximal}, Theorem 2.2], this group has order 
\(q-1\), so \(H_0\cong C_{(q-1)/2}\rtimes C_2\). Write \((q-1)/2=2^b m\), with \(m\) odd, and let \(N =C_m\) be the subgroup of order \(m\) that lies in $C_{(q-1)/2}$. Since \(N\) is characteristic in \(H_0\) and \(H=N_G(H_0)\), we have \(N\unlhd H\). In addition,
\[
|H:N|=|H:H_0||H_0:N|=2\cdot 2^{b+1}=2^{b+2}.
\]
This allows us  to assert that \(H\) is \(2\)-nilpotent. The case where \(H_0=D_{q+1}\) can be handled  similarly. Therefore, we find that the only maximal subgroup  of $G$ containing \(T\) is \(T\) itself, which is non-abelian simple, and those that do not contain $T$ are $2$-nilpotent. This means that \(G\) satisfies \((\ast_2)\), so part (3) of the theorem  is proved. 

\medskip
Finally, suppose that \(G=\mathrm{PGL}_2(q)\), where \(q=p^{2^a}\) with \(p\) an odd prime, \(a\geq0\), and \(p\equiv\pm1\pmod 8\) when \(a=0\). By [\citenum{Giudici2007PSL2Maximal}, Theorem 3.5], the maximal subgroups of \(G\) that do not contain \(T=\mathrm{PSL}_2(q)\) are
\[
C_p^{\,2^a}\rtimes C_{q-1},\qquad
D_{2(q-1)} \quad (q\neq 5),\qquad D_{2(q+1)},
\]
possibly \(S_4\) when \(q=p\equiv\pm3\pmod 8\), and possibly \(\mathrm{PGL}_2(q_0)\) when \(q=q_0^\ell\), with \(\ell\) an odd prime. The last two cases do not occur under these assumptions: if \(a=0\), then \(q=p\equiv\pm1\pmod 8\), whereas if \(a\geq1\), then \(q\) is not prime. But \(q=p^{2^a}\) cannot be written as \(q=q_0^\ell\) with \(\ell\) an odd prime. Thus, the only maximal subgroups of $G$ that do not contain \(T\) are \(C_p^{2^a}\rtimes C_{q-1}\), \(D_{2(q-1)}\) and \(D_{2(q+1)}\), all of which are \(2\)-nilpotent. Also, the only maximal subgroup of $G$ containing \(T\) is obviously \(T\) itself, which is non-abelian simple. This proves that \(G\) satisfies \((\ast_2)\), which gives case (2) and finishes the proof.
\end{proof}

\begin{proof}[Proof of Corollary F]
Suppose, for a contradiction, that \(G\) is a non-solvable group satisfying \((\ast_{2-2'})\). Then \(G\) also satisfies \((\ast_2)\). We distinguish two cases depending on whether \(G\) is simple or not.

\medskip
Case 1.  Assume first that \(G\) is non-simple. By Theorem E, and following its notation, we have either
\(
G\cong \mathrm{PGL}_2(q)
\)
or
\(
G\cong\langle \mathrm{PSL}_2(q),\delta\varphi^{2^{a-1}}\rangle,
\)
where \(q=p^{2^a}\) satisfies the corresponding arithmetic conditions. Since \(q\geq 5\) is odd, at least one of \(q-1\) and \(q+1\) has a non-trivial odd divisor.

Suppose first that \(G\cong \mathrm{PGL}_2(q)\). By [\citenum{Giudici2007PSL2Maximal}, Theorem 3.5], \(G\) has  maximal subgroups isomorphic to \(D_{2(q+\varepsilon)}\), for \(\varepsilon\in\{1,-1\}\). Such a subgroup is neither non-abelian simple nor \(2\)-decomposable whenever \(\varepsilon\in\{1,-1\}\) is chosen so that \(q+\varepsilon\) has a non-trivial odd divisor. This contradicts \((\ast_{2-2'})\).

Suppose now that
\(
G\cong\langle \mathrm{PSL}_2(q),\delta\varphi^{2^{a-1}}\rangle.
\)
By [\citenum{Giudici2007PSL2Maximal}, Theorem 1.5], \(G\) has a maximal subgroup
\(
H=N_G(D_{q+\varepsilon}),
\)
where \(\varepsilon\in\{1,-1\}\) is chosen so that \(q+\varepsilon\) has a non-trivial odd divisor. Set \(T=\mathrm{PSL}_2(q)\) and \(H_0=H\cap T\). Then \(H_0=D_{q+\varepsilon}\), and by [\citenum{Giudici2007PSL2Maximal}, Lemma 2.4], we have \(H/H_0\cong G/T\cong C_2\). In particular, \(H\) is not non-abelian simple. Moreover, \(H\) is not \(2\)-decomposable, since \(H_0\) is not. This again yields a contradiction.

\medskip
Case 2. Assume now that \(G\) is non-abelian simple. Then Theorem E gives
\(
G\cong \mathrm{PSL}_2(q),
\)
where \(q=p^f\geq 5\) is odd and satisfies the arithmetic conditions appearing
in (1). In particular, \(q\geq 11\). If \(q=11\), then Theorem~\ref{thm:psl2-maximals} provides a maximal subgroup
isomorphic to \(D_{12}\), which is neither non-abelian simple nor
\(2\)-decomposable. Thus, we may assume that \(q\geq 13\). By
Theorem~\ref{thm:psl2-maximals}, \(G\) has maximal subgroups isomorphic to
\(D_{q-1}\) and \(D_{q+1}\). Choose \(\varepsilon\in\{1,-1\}\) such that
\(q+\varepsilon\) has a non-trivial odd divisor. Then \(D_{q+\varepsilon}\)
is neither non-abelian simple nor \(2\)-decomposable.

\medskip
In either case, a contradiction is reached. As a result, \(G\) must be solvable as desired.
\end{proof}

\bigskip
\noindent
{\bf Acknowledgments}
The results of this research are part of the second author's Ph.D. thesis at the Universitat Jaume
I of Castell\'on. The first author is partially supported by the National Nature Science Fund (No. 12071181) of the People's Republic of China.

\medskip
\noindent
{\bf Data availability} Data sharing is not applicable to this article as no data sets were generated or analyzed during
the current study.

\bibliographystyle{plain}

\begin{thebibliography}{99}

\bibitem{BeltranShao2023Invariant}
Beltrán, A., Shao, C.G.:
New conditions on maximal invariant subgroups that imply solubility.
\newblock {\em Results Math.} {78}, no. 149, (2023).
\newblock doi:10.1007/s00025-023-01923-5.

\bibitem{BHRD}
Bray, J.N., Holt, D.F., Roney-Dougal, C.M.:
The Maximal Subgroups of the Low-Dimensional Finite Classical Groups.
\newblock London Math. Soc. Lecture Note Ser., vol. 407.
Cambridge University Press, Cambridge (2013).
\newblock doi:10.1017/CBO9781139192576.

\bibitem{Atlas}
Conway, J.H., Curtis, R.T., Norton, S.P., Parker, R.A., Wilson, R.A.:
Atlas of Finite Groups.
\newblock Oxford University Press, Oxford (1985).

\bibitem{DietrichLeePopielMonster}
Dietrich, H., Lee, M., Popiel, T.:
The maximal subgroups of the Monster.
\newblock {\em Adv. Math.} {469}, no. 110214 (2025).
\newblock doi:10.1016/j.aim.2025.110214.

\bibitem{GAP4}
GAP -- Groups, Algorithms, and Programming, Version 4.14.0
\newblock The  GAP  Group (2024).
\newblock \url{https://www.gap-system.org}.

\bibitem{Giudici2007PSL2Maximal}
Giudici, M.:
Maximal subgroups of almost simple groups with socle \(\mathrm{PSL}(2,q)\).
\newblock arXiv:math/0703685 [math.GR] (2007).

\bibitem{Huppert}
Huppert, B.:
Endliche Gruppen. I.
\newblock Springer-Verlag, Berlin, Heidelberg, New York (1967).

\bibitem{Jones}
Jones, G.A.:
Primitive permutation groups containing a cycle.
\newblock {\em Bull. Aust. Math. Soc.} {89}, no. 1, 159--165 (2014).
\newblock doi:10.1017/S000497271300049X.

\bibitem{KleidmanFi23}
Kleidman, P.B., Parker, R.A., Wilson, R.A.:
The maximal subgroups of the Fischer group \(\mathrm{Fi}_{23}\).
\newblock {\em J. London Math. Soc.} {(2) 39}, no. 1, 89--101 (1989).
\newblock doi:10.1112/jlms/s2-39.1.89.

\bibitem{KleidmanWilsonFi22}
Kleidman, P.B., Wilson, R.A.:
The maximal subgroups of \(\mathrm{Fi}_{22}\).
\newblock {\em Math. Proc. Camb. Phil. Soc.} {102}, no. 1, 17--23 (1987).
\newblock doi:10.1017/S0305004100067001.

\bibitem{KleidmanWilsonFi22Corr}
Kleidman, P.B., Wilson, R.A.:
Corrigendum: The maximal subgroups of \(\mathrm{Fi}_{22}\).
\newblock {\em Math. Proc. Camb. Phil. Soc.} {103}, 383 (1988).

\bibitem{KleidmanJ4}
Kleidman, P.B., Wilson, R.A.:
The maximal subgroups of \(J_4\).
\newblock {\em Proc. London Math. Soc.} {(3) 56}, no. 3, 484--510 (1988).
\newblock doi:10.1112/plms/s3-56.3.484.

\bibitem{KurzweilStellmacher}
Kurzweil, H., Stellmacher, B.:
The Theory of Finite Groups. An Introduction.
\newblock Universitext.
Springer, New York (2004).

\bibitem{LPS}
Liebeck, M.W., Praeger, C.E., Saxl, J.:
The maximal subgroups of the finite alternating and symmetric groups.
\newblock {\em J. Algebra} {111}, no. 2, 365--383 (1987).

\bibitem{LintonTh}
Linton, S.A.:
The maximal subgroups of the Thompson group.
\newblock {\em J. London Math. Soc.} {(2) 39}, no. 1, 79--88 (1989).
\newblock doi:10.1112/jlms/s2-39.1.79.

\bibitem{LintonFi24}
Linton, S.A., Wilson, R.A.:
The maximal subgroups of the Fischer groups \(\mathrm{Fi}_{24}\) and \(\mathrm{Fi}_{24}'\).
\newblock {\em Proc. London Math. Soc.} {(3) 63}, no. 1, 113--164 (1991).

\bibitem{Kourovka}
Mazurov, V.D., Khukhro, E.I. (eds.):
Unsolved Problems in Group Theory. The Kourovka Notebook, 19th edn.
\newblock Russian Academy of Sciences, Siberian Branch, Institute of Mathematics,
Novosibirsk (2018).

\bibitem{Monakhov}
Monakhov, V.S., Tyutyanov, I.N.:
On finite groups with given maximal subgroups.
\newblock {\em Siberian Math. Journal} {55}, no. 3, 451--456 (2014). \newblock doi:10.1134/S0037446614030069

\bibitem{Shao-Beltran}
Shao, C.G., Beltrán, A.:
Finite groups whose maximal subgroups are 2-nilpotent or normal.
\newblock {\em Bull. Malays. Math. Sci. Soc.} {47}, no. 142, (2024).

\bibitem{WilsonB}
Wilson, R.A.:
The maximal subgroups of the Baby Monster. I.
\newblock {\em J. Algebra} {211}, no. 1, 1--14 (1999).
\newblock doi:10.1006/jabr.1998.7601.

\end{thebibliography}

\end{document}